\documentclass{amsart}

\usepackage{nicefrac}
\usepackage{mathrsfs}
\usepackage{xcolor}

\newtheorem{theorem}{Theorem}[section]
\newtheorem{proposition}[theorem]{Proposition}
\newtheorem{lemma}[theorem]{Lemma}

\theoremstyle{definition}

\theoremstyle{remark}

\numberwithin{equation}{section}

\newcommand{\primos}{{\mathfrak {p}}}
\newcommand{\N}{\mathbb{N}}

\begin{document}

\title[Algebras and spaces of Dirichlet series with maximal Bohr's strip]{Algebras and Banach spaces of Dirichlet series with maximal Bohr's strip}

\author{Thiago R. Alves}

\address{Departamento de Matem\'{a}tica,
	Instituto de Ci\^{e}ncias Exatas,
	Universidade Federal do Amazonas,
	69.077-000 -- Manaus -- Brazil}

\email{alves@ufam.edu.br}

\thanks{The first author was supported in part by Coordenação de Aperfeiçoamento de Pessoal de Nível Superior - Brasil (CAPES) - Finance Code 001 and FAPEAM}

\author{Leonardo Brito}

\address{Departamento de Matem\'{a}tica,
	Instituto de Ci\^{e}ncias Exatas,
	Universidade Federal do Amazonas,
	69.077-000 -- Manaus -- Brazil}

\email{leocareiro2018@gmail.com}

\thanks{The second author was supported by FAPEAM}

\author{Daniel Carando}

\address{Departamento de Matem\'atica,
	Facultad de Cs. Exactas y Naturales,
	Universidad de Buenos Aires
	and IMAS-UBA-CONICET, Argentina.}

\email{dcarando@dm.uba.ar}

\thanks{The third author was supported by CONICET-PIP 11220130100329CO and  ANPCyT PICT 2018-04104.}


\date{}

\dedicatory{}


\begin{abstract}
	We study linear and algebraic structures in sets of Dirichlet series with maximal Bohr's strip. More precisely, we consider a set $\mathscr M$ of Dirichlet series which are uniformly continuous on the right half plane and whose strip of uniform but  not absolute convergence has maximal width, i.e.,  $\nicefrac{1}{2}$.  Considering the uniform norm, we show that $\mathscr M$ contains an isometric copy of $\ell_1$ (except zero) and is strongly $\aleph_0$-algebrable. Also, there is a dense $G_\delta$ set such that any of its elements generates a free algebra contained in $\mathscr M\cup \{0\}$. Furthermore, we investigate $\mathscr{M}$ as a subset of the Hilbert space of Dirichlet series whose coefficients are square-summable. In  this case, we prove that $\mathscr M$ contains an isometric copy of $\ell_2$ (except zero).
\end{abstract}

\maketitle

\section{Introduction and main results}

Mathematics is plenty of examples that seem to challenge the intuition. For instance, discontinuous additive functions,  Weierstrass' Monsters, Peano curves, non-extendable holomorphic functions, and so on and so forth. The counter-intuitiveness of these examples may lead us to believe they must be rare, but usually this is not the case. Moreover, recent investigations are presenting a very interesting picture. Many of these peculiar examples/objects are not only far away from being rare: in many situations, the set formed by these objects can even contain big linear or algebraic structures.  As a seminal example, Gurarij in \cite{Gu91} constructed infinite dimensional subspaces of $C([0,1])$  all whose nonzero elements are nowhere differentiable functions. Since then, a whole theory was built in this direction, especially in the last years. Many of these advances are documented in the recent monograph \cite{AronBerPelSeo} (see also \cite{AronMaestro, BerPelSeo}).

In this work we find different structures in the set of Dirichlet with maximal Bohr's strips (see below for the definition). Previous results on similar lines can be found in \cite{AlbJuaPab} for polynomials with bad convergence properties, in \cite{AlvCar} for  holomorphic functions with wild behaviour in certain points or in \cite{Bayart05-michigan,Bayart05-studia,BayaQuar} for different functions including Dirichlet series.
We refer the reader to \cite{DefGarMaeSev19} and \cite{Queffelec^2} for a most comprehensive background on functional analytic aspects of Dirichlet series and infinite dimensional holomorphy.

A {\it Dirichlet series} is a series of the form $\sum_{n=1}^\infty a_n n^{-s}$, where the coefficients $a_n$ are complex numbers and $s$ is a complex variable.
The natural domains of convergence of Dirichlet series are half-planes. Given a  Dirichlet series $D= \sum a_{n} n^{-s}$ we can consider three natural abscissas which define the biggest half-planes on which $D$ converges, converges uniformly and converges absolutely:
\begin{equation*} 
	\sigma_{c}(D) \leq \sigma_{u} (D) \leq \sigma_{a} (D) \, .
\end{equation*}
It is not hard to see that
\begin{equation*}
	\sup_{D \text{ Dir. ser.}} \sigma_{a} (D) - \sigma_{c} (D)  = 1.
\end{equation*}
Harald Bohr was among the first to consider the problem of finding the maximal width of the strip on which a Dirichlet series can converge uniformly but not absolutely (this strip is usually called \emph{Bohr's strip}). Thus, the so called \textit{Bohr's absolute convergence problem} \cite{Bo13,Bo13_Goett} was to determine the number
\[
S : = \sup_{D \text{ Dir. ser.}} \sigma_{a} (D) - \sigma_{u} (D) \,.
\]
Bohr \cite{Bo13_Goett} first showed in 1913 that $S \leq 1/2$\,,
and later in 1931 Bohnenblust and Hille \cite{BoHi31} proved that actually
\begin{equation} \label{BoBoHi}
	S = 1/2\,.
\end{equation}

In the sequel, we write $\mathbb{C}_0$ for the open right half plane and, more generally, for $a \in \mathbb R$  we set
$$\mathbb C_a=\{z\in \mathbb C\colon \text{Re}\, z > a\}.$$
Bohr's fundamental theorem (see \cite{Bo13} or \cite[Theorem 1.13]{DefGarMaeSev19}), ensures that every bounded holomorphic function $f : \mathbb{C}_0 \to \mathbb{C}$ which may be represented as a Dirichlet series in some half-plane converges uniformly on $\mathbb{C}_\delta$ for each $\delta > 0$. Let $\mathscr{H}_\infty$ denote the space of all such functions. It is well-known that $\mathscr{H}_\infty$ is actually a Banach algebra when equipped with the supremum norm. As a consequence of Bohr's results, the absolute convergence problem and its solution by Bohnenblust and Hille  can be written as
\begin{equation}\label{Svalue}
	S  = \sup_{D \in \mathscr{H}_\infty} \sigma_{a} (D) = \frac 1 2 \,.
\end{equation}
As expected, finding explicit Dirichlet series $D\in \mathscr{H}_\infty$ such that $\sigma_{a} (D) = \frac 1 2 $ is not an easy task. However, we will show that we have plenty of them and that the set of such series contains large linear and algebraic structures. Moreover, we can get all this in a much smaller subalgebra of  $\mathscr{H}_\infty$ which we now define.

A \textit{Dirichlet polynomial} is a Dirichlet series of the form $\sum_{n=1}^N a_n n^{-s}$. Let $\mathscr{A}(\mathbb{C}_0)$ denote the subalgebra of $\mathscr{H}_\infty$ of all functions $f$ which are uniform limits on $\mathbb{C}_0$ of a sequence of Dirichlet polynomials. It follows from \cite[Theorem 2.3]{AroBayGauMaeNes} that $f \in \mathscr{A}(\mathbb{C}_0)$ if and only if $f$ is represented by a Dirichlet series pointwise on $\mathbb{C}_0$ and $f$ is uniformly continuous on $\mathbb{C}_0$. This implies that each function $f \in \mathscr{A}(\mathbb{C}_0)$ extends uniformly to a uniformly continuous function on $\overline{\mathbb{C}_0}$.

Our goal is to study the following set of Dirichlet series:
$$\mathscr{M} := \big\{D \in \mathscr{A}(\mathbb{C}_0) : \sigma_a(D) = \displaystyle \frac{1}{2}\big\}.$$
Note that, by \eqref{Svalue}, the Dirichlet series belonging to $\mathscr{M}$ are those in $\mathscr{A}(\mathbb{C}_0)$ whose strip of uniform but not absolute convergence has maximal width.

{ Let $E$ be a topological vector space and $\kappa$ be a cardinal number. A subset $Z \subset E$ is said to be {\it $\kappa$-spaceable} if $Z \cup\{0\}$ contains a closed vector subspace of $E$ with dimension $\kappa$. Moreover, we say that a subset $Z \subset E$ is {\it maximal spaceable} if $Z$ is  $dim (E)$-spaceable.}

Our first main theorem regarding the set $\mathscr{M}$ is the following.
{
\begin{theorem} \label{thm-radius-spac}

The set $\mathscr{M}$ is maximal spaceable in  $\mathscr{A}(\mathbb{C}_0)$. More precisely, there is an isometric copy of $\ell_1$ in $\mathscr{A}(\mathbb{C}_0)$ which is contained in $\mathscr{M} \cup \{0\}$.
\end{theorem}
We also consider the space of Dirichlet series whose coefficients are square-summable:
\begin{align*}
	\mathscr{H}_2 := \left\{\sum_{n=1}^\infty a_n n^{-s} : \Big\| \sum_{n=1}^\infty a_n n^{-s}\Big\|_2 := \Big(\sum_{n=1}^\infty |a_n|^2\Big)^{1/2} < \infty\right\}.
\end{align*}
The space $\mathscr{H}_2$ is a Hilbert space with the inner product
$$
	\big\langle \sum_{n=1}^\infty a_n n^{-s} , \sum_{n=1}^\infty b_n n^{-s}\big\rangle := \sum_{n=1}^\infty a_n \overline{b_n}.
$$
In this setting we have a following notion of  $\mathscr{H}_2$-abscissa of a Dirichlet series $D = \sum_{n=1}^\infty a_n n^{-s}$.
\begin{align*}
	\sigma_{\mathscr{H}_2}(D) := \inf \left\{ \sigma \in \mathbb{R} : \sum_{n=1}^{\infty} \frac{a_n}{n^\sigma} n^{-s} \mbox{ belongs to } \mathscr{H}_2 \right\}.
\end{align*}
This notion is a natural counterpart of the abscissa of uniform convergence, since the abscissa of uniform convergence of a Dirichlet series $D = \sum_{n=1}^\infty a_n n^{-s}$ can also be reformulated as follows (\cite[Remark 1.23]{DefGarMaeSev19}):
\begin{align*}
	\sigma_u(D) = \inf \left\{ \sigma \in \mathbb{R} : \sum_{n=1}^{\infty} \frac{a_n}{n^\sigma} n^{-s} \mbox{ belongs to } \mathscr{H}_\infty \right\}.
\end{align*}
With this new abscissa, Bohr's absolute convergence problem for $\mathscr{H}_2$ becomes to determine the number
\[
S_2 : = \sup_{D \text{ Dir. ser.}} \sigma_{a} (D) - \sigma_{\mathscr{H}_2} (D) \,.
\]
Although the set $\mathscr{H}_2$ is bigger than $\mathscr{H}_\infty$, the solution for this problem is also $S_2 = \frac{1}{2}$ (see \cite[Remark 11.3]{DefGarMaeSev19}). With a translation argument, we can rewrite this fact as
\begin{align*}
	S_2 = \sup_{D \in \mathscr{H}_2} \sigma_a(D) = \dfrac{1}{2}.
\end{align*}
Note that, since $\mathscr{A}(\mathbb{C}_0) \subset \mathscr{H}_\infty \subset \mathscr{H}_2$, the set $\mathscr{M}$ is contained in $\mathscr{H}_2$. However, its spaceability as a subset of $\mathscr{H}_2$ is not a consequence of Theorem~\ref{thm-radius-spac}, since we have different norms and $\mathscr{A}(\mathbb{C}_0)$ is not closed in $\mathscr{H}_2$.
Our result concerning $\mathscr{H}_2$ reads as follows.

\begin{theorem} \label{thm-absc-spaceable-H2}

The set $\mathscr{M}$ is maximal spaceable in  $\mathscr{H}_2$. More precisely, there is an isometric copy of $\ell_2$ in $\mathscr{H}_2$ which is contained in $\mathscr{M} \cup \{0\}$.
\end{theorem}

As $\mathscr{A}(\mathbb{C}_0)$ is also a Banach algebra we may ask about algebrability of $\mathscr M$. Let us recall the precise definition of strongly algebrable sets which was introduced and coined in \cite{BarGla13}.} Let $X$ be an arbitrary set, $\mathcal{A}$ be an algebra of functions $f \colon X \to \mathbb{C}$ and $\kappa$ be a cardinal number. A subset $Z \subset \mathcal{A}$ is said to be {\it strongly $\kappa$-algebrable} if there is a sub-algebra $\mathcal{B}$ of $\mathcal{A}$ which is generated by an  algebraically independent set of generators with cardinality $\kappa$ and such that $\mathcal{B} \subset Z \cup \{ 0 \}$. With this, we can state our last main theorem.
\begin{theorem} \label{thm-Gdelta}
	The set $\mathscr{M}$ is strongly $\mathfrak{\aleph_0}$-algebrable as a subset of $\mathscr{A}(\mathbb{C}_0)$. Also, there is a dense $G_\delta$ subset of $\mathscr{A}(\mathbb{C}_0)$ such that any of its elements generates a free algebra contained in $\mathscr{M}\cup \{0\}$.
\end{theorem}
The proofs of our main theorems are carried out in {Sections \ref{sec-spaceab}, \ref{Sec.H2} and \ref{Sec.alg-of-M}}. Some technical results used in these proofs are stated and proved in Section \ref{sec-techlem}.

\section{Proof of Theorem \ref{thm-radius-spac}} \label{sec-spaceab}

Let us fix some notations. Let $\mathbb{N}_0$ denote the set of nonnegative integers and $\mathbb{N}_0^{(\mathbb{N})}$ the set of non-negative multi-indices (i.e., sequences of non-negative integers with finite nonzero elements). For any $\alpha = (\alpha_1, \alpha_2, \alpha_3, \ldots) \in \mathbb{N}_0^{(\mathbb{N})}$ we set $supp(\alpha) := \{j \in \mathbb{N} : \alpha_j \not= 0\}$ and $|\alpha| := \sum_j \alpha_j$. Also, for a sequence $z=(z_j)_{j=1}^\infty$, we write
$$z^\alpha := z_1^{\alpha_1} z_2^{\alpha_1} z_3^{\alpha_3} \cdots,$$
where the product is finite since $\alpha$ has finite length.

We now define the so-called \emph{Bohr transform}, which links Dirichlet series with power series in infinitely many variables (first formally, then in a precise way which is fundamental for our purposes). Let  $\mathfrak{p}_k$ denote the $k$-th prime number and write $\mathfrak{p} := (\mathfrak{p}_k)_{k=1}^\infty$.
Each $n \in \N$ has a unique representation of the form $n= \primos^{\alpha} = \mathfrak{p}_1^{\alpha_1}\cdot \cdots \cdot \mathfrak{p}_m^{\alpha_m}$. This one-to-one correspondence $n \in \N \leftrightarrow \alpha \in \N_0^{(\N)}$ allows us to define
\begin{align*}
	\mathfrak{B}: \mathfrak{P} \quad &\xrightarrow{\phantom{ a_n = a_{\primos^{\alpha}}=c_{\alpha}}} \quad \mathfrak{D} \\
	\sum_{\alpha} c_{\alpha}  z^{\alpha}&\xrightarrow{ a_n = a_{\primos^{\alpha}}=c_{\alpha}} \sum_{n=1}^{\infty} \frac{a_{n}}{n^s},
\end{align*}
between the space $\mathfrak{P}$ of formal power series and the space $\mathfrak{D}$ of formal Dirichlet series. The function $\mathfrak{B}$ is called the Bohr transform and is an algebra isomorphism.

It follows from \cite[Theorem 2.19]{DefGarMaeSev19} that any holomorphic function $f : B_{c_0} \to \mathbb{C}$ has a monomial series representation at $0$. That is, there exists a unique family of coefficients $(c_\alpha(f))_{\alpha \in \mathbb{N}_0^{(\mathbb{N})}}$ such that for every $z \in B_{c_{00}}$,
\begin{align} \label{monomial rep.}
	f(z) = \sum_{\alpha \in \mathbb{N}_0^{(\mathbb{N})}} c_\alpha(f) z^\alpha.
\end{align}
The coefficients $c_\alpha(f)$ are called {\it monomial coefficients} of $f$. As consequence, the Banach algebra $H_{\infty}(B_{c_0})$ of bounded holomorphic functions on $B_{c_0}$ can be thought of as a subset of $\mathfrak{P}$. As we can see in  \cite[Theorem 3.8]{DefGarMaeSev19}, Bohr transform is an isometric isomorphism between $H_{\infty}(B_{c_0})$ and $\mathscr{H}_\infty$.

We recall that a continuous $m$-homogeneous polynomial from $c_0$ into $\mathbb{C}$ is the restriction to the diagonal of a
continuous $m$-linear functional. The symbol $\mathscr{P}_m(c_0)$ denotes the vector space of all continuous $m$-homogeneous
polynomials from $c_0$ into $\mathbb{C}$. Since every continuous homogeneous polynomial is a holomorphic function,  we may write any $P \in \mathscr{P}_m(c_0)$ as in (\ref{monomial rep.}) for every $z \in c_{00}$.

The proof of the following lemma is based on the proofs of \cite[Propositions 4.2 and 4.6]{DefGarMaeSev19} (see also \cite{AlbJuaPab}).

\begin{lemma} \label{lem-spac}
	Given $\Theta\subset \mathbb N$ containing an infinite arithmetic progression and  $m \geq 2$, there is an $m$-homogeneous polynomial $P \in \mathscr{P}_m(c_0)$ supported on $\Theta$ (i.e.,  $c_\alpha(P)=0$ whenever $supp(\alpha)\not\subset \Theta$) and such that for every $\varepsilon > 0$, we have
	\begin{align*}
		\sum_{\alpha} |c_\alpha(P)| \frac{1}{(\mathfrak{p}^\alpha)^{\frac{1}{(\frac{2m}{m-1} + \varepsilon)(1 + \varepsilon)}}} = \infty.
	\end{align*}
\end{lemma}

\begin{proof} We may suppose that $\Theta$ is actually an infinite arithmetic progression: if the result holds for such sets, it clearly holds for larger ones. So we take $\Theta= \{u + k v : k \in \mathbb{N}\}$ for some $u,v \in \mathbb{N}$
	We fix  a prime number $p>m$ and subdivide $\Theta$ in blocks:
	$$\Theta = \big\{ {{n^{1}_{1}},\dots,{n^{p}_{1}}},{n^{1}_{2}},\dots,{n^{p^{2}}_{2}},
	{n^{1}_{3}},\dots,{n^{p^{3}}_{3}},{n^{1}_{4}},\dots,{n^{p^{4}}_{4}},
	{n^{1}_{5}},\dots,{n^{p^{5}}_{5}},\dots    \big\}.$$
	Now, for every $k \in \mathbb{N}$ let $B^{(k)}$ denote the subset of $\Theta$ of indices lying in the $k$th block, i.e.,
	$$B^{(k)}=\big\{ {{n^{1}_{k}},\dots,{n^{p^k}_{k}}}  \big\}.$$
	With the previous notations in mind, we define the contractions
	$\Pi_{k}:c_{0}\longrightarrow \ell_\infty^{p^{k}} $ by
	$$\Pi_k(z) := \sum_{ j=1}^{p^{k}} z_{n^{j}_{k}} e_j.$$
	
	From \cite[Lemma 4.7]{DefGarMaeSev19} we can take a sequence $(R_k)_{k=1}^\infty$ of continuous $m$-homogeneous polynomials satifying
	$$\|R_k : \ell_\infty^{p^k} \to \mathbb{C}\|_\infty \leq p^{k \frac{m+1}{2}}, \ \ k \in \mathbb{N},$$
	and
	$$\eta := \inf\{|c_\alpha(R_k)| : k \in \mathbb{N}, \alpha \in \mathbb{N}_0^{p^k}, |\alpha| = m\} > 0.$$
	Now, for each $k \in \mathbb N$, we define $Q_{k}:c_{0}\longrightarrow\mathbb{C}$ by  $$Q_{k}(z) := \frac{1}{k^{2}}p^{-k\frac{m+1}{2}}(R_{k}\circ \Pi_{k})(z),$$
	which is a well-defined $m$-homogeneous polynomial with  $\|Q_k\| \leq 1/k^2$. Moreover, by construction, we have $c_\alpha(Q_k)\ne 0$ if and only if $supp (\alpha)\subset B^{(k)}$.
	Thus, we can (finally) define the desired continuous $m$-homogeneous polynomial $P : c_{0} \longrightarrow \mathbb{C}$  by
	$$P(z) := \sum_{k=1}^{\infty}Q_{k}(z).$$
	Note that $\|P\| \leq \pi^2 / 6$.	Let us prove that $P$ satisfies the requirements of the lemma. First, since the sets $B^{(k)}$s are pairwise disjoint, it follows that \begin{equation}\label{c-alphas-nonzero}
		c_\alpha(P)\ne 0 \quad\text{if and only if}\quad  supp(\alpha) \subset B^{(k)} \text{ for some }k\in \mathbb N.
	\end{equation}
	In particular, $c_\alpha(P)=0$ whenever $supp(\alpha)\not\subset \Theta$.
	
	We now let $\varepsilon > 0$ and choose  $0<\delta<1$ satisfying $\frac{2m}{m-1}+\varepsilon=\frac{2m}{m-1}\frac{1}{1-\delta}$. We also take $0<b<1$ such that $p^{\delta}b^{1-\delta}>1$. Define $w = (w_{\ell})_{\ell = 1}^\infty$  by
	\[
	w_{\ell} :=
	\begin{cases}
		\left(\dfrac{b}{p}\right)^{k \frac{m-1}{2m} (1-\delta)} & \text{if $\ell \in B^{(k)}$};\\
		0 & \text{otherwise.}
	\end{cases}
	\]
	It is easy to check that  $\lim\limits_{\ell \to \infty} w_\ell = 0$, i.e., $w\in c_0$. By the uniqueness of the monomial coefficients, we also have
	
	\begin{align*}
		\sum_{|\alpha| = m} | c_\alpha(P) w^\alpha | & =  \sum_{k=1}^{\infty}\frac{1}{k^{2}}p^{-k\frac{m+1}{2}} \sum_{\substack{\alpha \in \mathbb{N}_0^{(\mathbb{N})} \\ \vert\alpha\vert=m \\ supp(\alpha)\subset B^{(k)}}} \vert c_{\alpha}\left(R_{k}\circ \Pi_k\right)w^{\alpha}\vert
		\\ &= \sum_{k=1}^{\infty}\frac{1}{k^{2}}p^{-k\frac{m+1}{2}} \sum_{\substack{\beta \in \mathbb{N}_0^{p^k} \\ \vert\alpha\vert=m }} \vert c_{\beta}\left(R_{k}\right)(\Pi_k(w))^{\beta}\vert
		\\ &= \sum_{k=1}^{\infty}\frac{1}{k^{2}}p^{-k\frac{m+1}{2}} \sum_{\substack{\beta \in \mathbb{N}_0^{p^k} \\ \vert\alpha\vert=m }} \vert c_{\beta}\left(R_{k}\right)\vert \left( \frac{b}{p} \right)^{k(m-1)(1-\delta)/2}.
	\end{align*}
	Proceeding as in \cite[p. 105]{DefGarMaeSev19}, we obtain		
	\begin{align} \label{eq-of-nonanaly}
		\sum_{|\alpha| = m}|c_\alpha(P)w^\alpha|  &\geq \frac{\eta}{m!} \sum_{k = 1}^\infty \frac{1}{k^2}(p^\delta b^{1-\delta})^{\frac{m-1}{2}k}=+\infty \,
	\end{align}
	since $p^\delta b^{1-\delta} > 1$. We see also that  $w \in \ell_{\frac{2m}{m-1}+\varepsilon}$, since
	\begin{align*}
		\sum_{j = 1}^\infty (w_{j})^{\frac{2m}{m-1}\frac{1}{1 - \delta}} = \sum_{k=1}^\infty \sum_{\ell = 1}^{p^k} \left(\frac{b}{p}\right)^{k \frac{m-1}{2m}(1-\delta)\frac{2m}{m-1} \frac{1}{1-\delta}} = \sum_{k=1}^\infty b^k < \infty.
	\end{align*}
	Since the sequence $\left(w_{u + k v}\right)_{k = 1}^\infty$ is decreasing, for each  $k \in \mathbb{N}$ we have
	\begin{align*}
		w_{u + k v} (u + k v)^{\frac{1}{{(\nicefrac{2m}{m-1})+\varepsilon}}} & \leq \bigg(\left(\dfrac{u}{k}+v\right) \, \sum_{\ell = 1}^{k} \left(w_{u + \ell v}\right)^{\frac{2m}{m-1}+\varepsilon}\bigg)^{\frac{1}{\frac{2m}{m-1}+\varepsilon}} \\ & \leq (u+v)^{\frac{1}{{(\nicefrac{2m}{m-1})+\varepsilon}}} \Vert w\Vert_{\frac{2m}{m-1}+\varepsilon}=:K.
	\end{align*}
	Recalling that  $w_\ell=0$ for $\ell\not\in \Theta$, we actually obtain
	$$w_\ell \, \ell^{\frac{1}{{(\nicefrac{2m}{m-1})+\varepsilon}}} \leq K$$
	for all $\ell \in \mathbb N$. By the prime number theorem, there is a constant $C>0$ such that  ${p}_{n}\leq Cn^{1+\varepsilon}$ for every $n\in\mathbb{N}$. Thus, for any $\alpha = (\alpha_1, \ldots, \alpha_N, 0, 0, \ldots) \in \mathbb{N}_{0}^{(\mathbb{N})}$ with $\vert \alpha\vert=m$, we have
	\begin{eqnarray*}			\frac{1}{\left(\mathfrak{p}^{\alpha}\right)^\frac{1}{{\left(\frac{2m}{m-1}+\varepsilon\right)(1+\varepsilon)}}}&=&\frac{1}{\left(\mathfrak{p}_{1}^{\alpha_{1}}\mathfrak{p}_{2}^{\alpha_{2}} \cdots \mathfrak{p}_{N}^{\alpha_{N}}\right)^\frac{1}{{\left(\frac{2m}{m-1}+\varepsilon\right)(1+\varepsilon)}}}
		\nonumber	\\
		&\geq&\frac{1}{\left(\left(C \, 1^{1+\varepsilon}\right)^{\alpha_{1}}\left(C \, 2^{1+\varepsilon}\right)^{\alpha_{2}} \cdots \left(C \, N^{1+\varepsilon}\right)^{\alpha_{N}}\right)^\frac{1}{{\left(\frac{2m}{m-1}+\varepsilon\right)(1+\varepsilon)}}}
		\nonumber	\\ &\geq& { \frac{1}{C^{\frac{m}{\left(\frac{2m}{m-1}+\varepsilon\right)(1+\varepsilon)}}}\left(\frac{w_{1}}{K}\right)^{\alpha_{1}}\left(\frac{w_{2}}{K}\right)^{\alpha_{2}} \cdots \left(\frac{w_{N}}{K}\right)^{\alpha_{N}}}
		\nonumber	\\ &=& {\frac{1}{C^{\frac{m}{\left(\frac{2m}{m-1}+\varepsilon\right)(1+\varepsilon)}}K^{m}}w^{\alpha}.}
	\end{eqnarray*}
	This gives
	\begin{align*}
		\sum_{|\alpha| = m} |c_\alpha(P)| \frac{1}{(\mathfrak{p}^\alpha)^{\frac{1}{(\frac{2m}{m-1} + \varepsilon)(1 + \varepsilon)}}} \geq \dfrac{1}{C^{\frac{m}{\left(\frac{2m}{m-1}+\varepsilon\right)(1+\varepsilon)}}K^{m}} \cdot {\sum_{|\alpha| = m} |c_\alpha(P){w}^\alpha|},
	\end{align*}
	and hence the proof follows from $(\ref{eq-of-nonanaly})$.
\end{proof}

Now we are ready to prove Theorem \ref{thm-radius-spac}.

\begin{proof}[\underline{Proof of Theorem \ref{thm-radius-spac}}]
	First note that we can choose a countably infinite family of  $\Theta$s as in the previous Lemma that are pairwise disjoint.  Let us write such a family as  $\{\Theta_{k,m} : k,m \in\mathbb{N}\}$. It follows from Lemma \ref{lem-spac} that for every $m \geq 2$ and every $k \in \mathbb{N}$, there is $P_{k,m} \in \mathscr{P}_m(c_0)$ such that, for every $\varepsilon > 0$,
	\begin{align} \label{Eq-spac-inf}
		\sum_{|\alpha| = m} |c_\alpha(P_{k,m})| \frac{1}{(\mathfrak{p}^\alpha)^{\frac{1}{(\frac{2m}{m-1} + \varepsilon)(1 + \varepsilon)}}} = \infty
	\end{align}
	and that $c_\alpha(P_{k,m}) = 0$ if $supp(\alpha) \not\subset \Theta_{k,m}$ (that is, $P_{k,m}$ is supported on $\Theta_{k,m}$). We define the mapping
	\begin{align}
		T : (\lambda_k)_{k=1}^\infty \in \ell_1 \mapsto \sum_{k=1}^\infty \lambda_k \sum_{m=2}^\infty \frac{1}{2^{m-1}} \frac{P_{k,m}}{{\|P_{k,m}\|}_\infty} \in H_\infty(B_{c_0}).
	\end{align}
	Note the $T$ is a well-defined linear operator. To complete the proof, we only need to show that $T$ is an isometry and $\sigma_a((\mathfrak{B} \circ T)(\lambda_k)) = 1/2$ for each $(\lambda_k) \in \ell_1 \setminus \{0\}$.
	
	To see that $T$ is an isometry, we first note that
	\begin{align} \label{Des-isom}
		\|T((\lambda_k)_{k=1}^\infty)\|_\infty = \left\|\sum_{k=1}^\infty \lambda_k \sum_{m=2}^\infty \frac{1}{2^{m-1}} \frac{P_{k,m}}{{\|P_{k,m}\|}_\infty}\right\|_\infty \leq \sum_{k=1}^\infty |\lambda_k|.
	\end{align}
	Now, we take $N, M \in \mathbb{N}$ with $M\ge 2$. We write $\lambda_k = |\lambda_k| e^{i\theta_k}, 0 \leq \theta_k < 2 \pi$, for $k = 1,2, \ldots, N$. Since $P_{k,m}$ is supported on $\Theta_{k,m}$, for each $k = 1, \ldots, N$ and each $m = 2, \ldots, M$, we may find $(z_{k,m}^\ell)_{\ell = 1}^\infty \subset B_{c_0}$ with $supp(z_{k,m}^\ell) \subset \Theta_{k,m}$ and $\lim\limits_{\ell \to \infty} P_{k,m}(z_{k,m}^\ell) = {\|P_{k,m}\|}_\infty e^{-i\theta_k}$. Hence, for $z_\ell := \sum_{k=1}^N \sum_{m=2}^M z_{k,m}^\ell \in B_{c_0}$, we get
	\begin{align*}
		\left\|\sum_{k=1}^\infty \lambda_k \sum_{m=2}^\infty \frac{1}{2^{m-1}} \frac{P_{k,m}}{{\|P_{k,m}\|}_\infty}\right\|_\infty &\geq \left|\sum_{k=1}^\infty \lambda_k \sum_{m=2}^\infty \frac{1}{2^{m-1}} \frac{P_{k,m}(z_\ell)}{{\|P_{k,m}\|}_\infty}\right| \\ & \hspace{-1in} = \left|\sum_{k=1}^N \lambda_k \sum_{m=2}^M \frac{1}{2^{m-1}} \frac{P_{k,m}(z_{k,m}^\ell)}{{\|P_{k,m}\|}_\infty}\right| \xrightarrow{\ell \to \infty} \sum_{k=1}^N |\lambda_k| \left(1 - \frac{1}{2^{M-1}}\right).
	\end{align*}
	Since $M$ is arbitrary, we get equality in (\ref{Des-isom}) and $T$ is an isometry.
	
	Now, we prove that $\sigma_a((\mathfrak{B} \circ T)(\lambda_k)) = 1/2$ for each $(\lambda_k) \in \ell_1 \setminus \{0\}$.	We take  $\varepsilon > 0$ and a nonzero $(\lambda_k) \in \ell_1$, and choose $k_0 \in \mathbb{N}$ with $\lambda_{k_0} \not= 0$. Also, we write $$(\mathfrak{B} \circ T)(\lambda_k) = \sum_{n=1}^\infty a_n n^{-s}.$$ Since the sets $\Theta_{k,m}$ are pairwise disjoint, the polynomials $P_{k,m}$ have mutually disjoint supports. It follows from \cite[Theorem 3.8]{DefGarMaeSev19}, \cite[Theorem 2.19]{DefGarMaeSev19} and (\ref{Eq-spac-inf}) that for every $m \geq 2$ and every $\varepsilon > 0$ we have
	\begin{align*}
		\sum_{n=1}^\infty |a_n| \frac{1}{n^{\frac{1}{(\frac{2m}{m-1} + \varepsilon)(1 + \varepsilon)}}} &= \sum_{k=1}^\infty \sum_{\ell=2}^\infty \sum_{\underset{supp(\alpha)\subset\Theta_{k,\ell}}{|\alpha| = \ell}} |\lambda_k| 2^{1-\ell} \frac{|c_\alpha(P_{k,\ell})|}{{\|P_{k,\ell}\|}_\infty} \frac{1}{(\mathfrak{p}^\alpha)^{\frac{1}{(\frac{2m}{m-1} + \varepsilon)(1 + \varepsilon)}}} \\ &\geq \sum_{\underset{supp(\alpha)\subset\Theta_{k_{0},m}}{|\alpha| = m}} |\lambda_{k_0}| 2^{1-m} \frac{|c_\alpha(P_{k_0,m})|}{{\|P_{k_0,m}\|}_\infty} \frac{1}{(\mathfrak{p}^\alpha)^{\frac{1}{(\frac{2m}{m-1} + \varepsilon)(1 + \varepsilon)}}} = \infty.
	\end{align*}
	Since $m\in \N$ and $\varepsilon>0$ are arbitrary, this proves that $\sigma_a((\mathfrak{B} \circ T(\lambda_k))) = 1/2$.
\end{proof}

{ \section{Proof of Theorem \ref{thm-absc-spaceable-H2}} \label{Sec.H2}}

We could prove Theorem \ref{thm-absc-spaceable-H2} proceeding as in the proof of Theorem \ref{thm-radius-spac}, making use of isometric identification between $\mathscr H_2$ and the Hardy space $H_2(\mathbb T^\infty)$ of functions in the infinite polytorus. However, we prefer to give a different proof, working directly on the Dirichlet series framework. On the one hand, in this way we do not need to define Hardy spaces in the infinite polytorus and the mentioned identification. On the other hand, we take the opportunity to introduce some concepts and notions that are necessary in the following sections.

Let us fix some notations. For a complex number $s \in \mathbb{C}$, we always let $s := \sigma + it$. Also, for every Dirichlet series $D = \sum_{n=1}^\infty a_n n^{-s}$, we set $$A_N(D,s) := \sum_{n=1}^N |a_n| n^{-s}.$$ Note that for $\delta>0$ we have that
\begin{equation}\label{eq-ANimpliessigma}
	\lim_{N\to +\infty}A_N(D,\delta) =+\infty \quad\text{implies}\quad\sigma_a(D) \ge \delta.
\end{equation}

We may formally write
$$\sum_{n=1}^\infty a_n n^{-s} = \sum_{m=0}^\infty \sum_{\Omega(n) = m} a_n n^{-s},$$
where the function $\Omega(n)$ counts the number of prime divisors of $n$, counted with multiplicity. Let $\mathfrak{D}_m$ be the set of this Dirichlet series for which $a_n \not= 0$ only if $\Omega(n) = m$. The series in $\mathfrak{D}_m$ are called \emph{$m$-homogeneous}. We also denote $\mathscr{H}^m_\infty := \mathscr{H}_\infty \cap \mathfrak{D}_m$ the closed subspace of $\mathscr{H}_\infty$ consisting of all $m$-homogeneous Dirichlet series. By means of Bohr's transform $\mathscr{P}_m(c_0)$ is isometrically isomorphic to $\mathscr{H}^m_\infty$ (see \cite[Theorem 3.12]{DefGarMaeSev19}). Hence, since $\mathscr{A}(\mathbb{C}_0)$ is isometrically isomorphic to the algebra of uniformly continuous holomorphic functions on $B_{c_0}$ via Bohr's transform (see \cite[Theorem 2.5]{AroBayGauMaeNes}), we have $\mathscr{H}_\infty^m \subset \mathscr{A}(\mathbb{C}_0)$.

For   $D= \sum_{n=1}^\infty a_n n^{-s}$ let $D^{(m)}$ be its $m$-homogeneous part
$$ D^{(m)}= \sum_{\Omega(n) = m} a_n n^{-s}.$$
It is clear that \begin{equation}\label{AN-homogeneo}
	A_N(D,\delta) \ge A_N(D^{(m)},\delta)
\end{equation}
for every $\delta >0$. Combining this with \eqref{eq-ANimpliessigma} we see that $\sigma_a(D)\ge \sigma_a(D^{(m)})$ for all Dirichlet series~$D$.

For $\Theta \subset \mathbb N$ we set
$$\mathscr{A}_\Theta(\mathbb{C}_0) := \big\{D = \sum_{n = 1}^\infty a_n n^{-s} \in \mathscr{A}(\mathbb{C}_0) : a_n = 0 \mbox{ if } \mathfrak{p}_i | n \mbox{ for some } i \not\in \Theta\big\}.$$
Note that the set $\mathscr{A}_\Theta(\mathbb{C}_0)$ is closed under sum and Dirichlet multiplication, and so it is a subalgebra of $\mathscr{A}(\mathbb{C}_0)$. Also, since ``taking $n$th coefficient'' is continuous (see \cite[Proposition 1.19]{DefGarMaeSev19})  $\mathscr{A}_\Theta(\mathbb{C}_0)$ is closed subalgebra of $\mathscr{A}(\mathbb{C}_0)$.

The following lemma is essentially a consequence of Lemma \ref{lem-spac} through the Bohr transform.
\begin{lemma}\label{Lemma-A}
	For every infinite subset $\Theta \subset \mathbb{N}$ containing an infinite arithmetic progression and every $M > m \geq 2$ there is $D \in \mathscr{H}_\infty^M \cap \mathscr{A}_\Theta(\mathbb{C}_0)$ satisfying
	$$\lim\limits_{N \to \infty} A_N(D, \delta_m) = \infty \ \ \mbox{with} \ \  \delta_m := \dfrac{m-1}{2m}.$$
\end{lemma}
\begin{proof}
	From Lemma \ref{lem-spac} there is $P\in\mathscr{P}_{M}(c_{0})$ such that, for every $\varepsilon>0$,
	\begin{align*}
		\sum_{\alpha} |c_\alpha(P)| \frac{1}{(\mathfrak{p}^\alpha)^{\frac{1}{(\frac{2M}{M-1} + \varepsilon)(1 + \varepsilon)}}} = \infty.
	\end{align*}
	Moreover, $c_\alpha(P) = 0$ if $supp(\alpha) \not\subset \Theta$. Take $D_1=\mathfrak B(P)$. Since $D_{1}(s)=\sum_{n = 1}^\infty a_{n}n^{-s}\in\mathscr{H}_{\infty}^{M}$ with $a_n=c_{\alpha}(P)$ for $n=\mathfrak{p}^{\alpha}$, it is clear that  $a_{n}=0$ if $\mathfrak{p}_{i} \vert n$ for some $i \in \Theta$. In particular, $D_1 \in \mathscr{A}_\Theta(\mathbb{C}_0)$. Again, from Bohr's transform, we get
	\begin{eqnarray*}
		\sum_{n = 1}^\infty \frac{\vert a_{n}\vert}{n^{\frac{1}{(\frac{2M}{M-1} + \varepsilon)(1 + \varepsilon)}}}=\sum_{n = 1}^\infty \vert a_{n}\vert\frac{1}{(\mathfrak{p}^\alpha)^{\frac{1}{(\frac{2M}{M-1} + \varepsilon)(1 + \varepsilon)}}}=\sum_{\alpha}\vert c_{\alpha}(P)\vert\frac{1}{(\mathfrak{p}^\alpha)^{\frac{1}{(\frac{2M}{M-1} + \varepsilon)(1 + \varepsilon)}}}=\infty
	\end{eqnarray*}
	for every $\varepsilon > 0$. This shows that $\lim\limits_{N \to \infty} A_N(D_1, \delta_M - \eta) = \infty$ for every $\eta > 0$. But since $\delta_M > \delta_m$, we must have $\lim\limits_{N \to \infty} A_N(D_1, \delta_m) = \infty$.
\end{proof}

{	
	
	Now we prove Theorem \ref{thm-absc-spaceable-H2}.
	
\begin{proof}[\underline{Proof of Theorem \ref{thm-absc-spaceable-H2}}]
	Let $\{\Theta_{k,m} : k, m \in \mathbb N\}$ be a family of pairwise disjoint subsets of $\mathbb{N}$ so that each $\Theta_{k,m}$ contains an infinite arithmetic progression. Normalizing in $\mathcal H^2$ the Dirichlet series given by Lemma \ref{Lemma-A}, we can take for each $m \geq 3$ and $k\in \mathbb N$ some $$D_{k,m} \in \mathscr{H}_\infty^{m} \cap \mathscr{A}_{\Theta_{k,m}}(\mathbb{C}_{0})\subset \mathscr{H}_{2}$$ such that $$\|D_{k,m}\|_2 = 1 \ \ \mbox{and} \ \ \sigma_{a}(D_{k,m}) \geq \dfrac{m-1}{2m}.$$
We define the mapping
	\begin{align*}
		T : (\lambda_k)_{k} \in \ell_2 \mapsto \sum_{k=1}^\infty \lambda_k \sum_{m=3}^{\infty}\dfrac{D_{k,m}}{2^{\frac{m-2}{2}}} \in \mathscr{H}_{2}.
	\end{align*}
	We claim that $T$ is a well-defined linear isometry. Indeed, the set $\{D_{k,m} : k\in \mathbb N, m \geq 3\}$ is orthonormal in $\mathscr{H}_2$ (recall that $\Theta_{k,m}s$ are pairwise disjoint) and, then:
	\begin{align*}
		\Big\Vert\sum_{k=1}^\infty \lambda_k \sum_{m=3}^{\infty}\dfrac{D_{k,m}}{2^{\frac{m-2}{2}}}\Big\Vert_{\mathscr{H}_{2}}^2& = \sum_{k=1}^{\infty}\vert \lambda_{k}\vert^{2}\sum_{m=3}^{\infty}\dfrac{1}{2^{m-2}}\Vert D_{k,m}\Vert_{\mathscr{H}_{2}}^{2} = \Vert(\lambda_{k})_{k}\Vert_{\ell_{2}}^2.
	\end{align*}
	On the other hand, for $(\lambda_{k})_{k}\in \ell_{2}\setminus\{0\}$ there is $k_{0} \in \mathbb N$ such that $\lambda_{k_{0}}\neq 0$ and hence
	$$\sigma_{a}\left(T((\lambda_k)_k)\right)\geq \sigma_{a}(D_{k_{0},m})=\dfrac{m-1}{2m} \xrightarrow{m \to \infty} \dfrac{1}{2},$$
	which means that $T(\ell_2) \setminus \{0\} \subset \mathscr{M}$ and completes the proof.
\end{proof}
}

	\section{Proof of Theorem \ref{thm-Gdelta}} \label{Sec.alg-of-M}

From now on, for every $\lambda = (\lambda_1, \ldots, \lambda_k) \in \mathbb{C}^k$ and $D \in \mathscr{H}_\infty$ we set $$D_\lambda := \lambda_1 D + \lambda_2 D^2 + \cdots + \lambda_k D^k.$$ Also, for every $\Theta \subset \mathbb{N}$ and every natural numbers $j,k,\ell$ and $m \geq 2$, we set $\delta_m := \dfrac{m-1}{2m}$ and define
\begin{align*}
	\mathscr{D}_\Theta(j, k, \ell, m):= \{D \in \mathscr{A}_\Theta (\mathbb {C}_0) :  &\textrm{ for every } \lambda=(\lambda_{1}, \ldots, \lambda_k) \in \mathbb{C}^k
	\textrm { satisfying } \\ &\|\lambda\|_\infty \leq j \ \ \mbox{and} \ \ |\lambda_k| \geq j^{-1}, \textrm { there is } N \in \mathbb {N} \\ &\textrm{ such that } A_N(D_\lambda, \delta_m)> \ell \}.
\end{align*}
These sets are inspired in those built in the proof of \cite[Theorem 5]{BayaQuar}.	The technical Lemmas \ref{lem-alg02} and \ref{lem-alg01}, which we state and prove in Section~\ref{sec-techlem}, show that the sets $\mathscr{D}_\Theta(j, k, \ell, m)$ are open and dense in $ \mathscr{A}_\Theta(\mathbb{C}_0)$. With these facts, we can prove a result which will be useful for proving Theorem \ref{thm-Gdelta}.

\begin{proposition} \label{pps1}
	For any infinite subset $\Theta \subset \mathbb N$ containing an infinite arithmetic progression, there exists a dense $G_\delta$ subset of $\mathscr{A}_\Theta(\mathbb{C}_0)$
	such that any of its elements generates a free algebra contained in $(\mathscr{M}\cap\mathscr{A}_\Theta(\mathbb{C}_0)) \cup \{0\}$.
\end{proposition}
\begin{proof}
	It follows from Baire category theorem and Lemmas \ref{lem-alg02} and \ref{lem-alg01} that the set
	$$\mathscr{D}_{\Theta}:=\displaystyle\bigcap_{j\geq 1}\bigcap_{k\geq 1}\bigcap_{\ell\geq 1}\bigcap_{m\geq 2}\mathscr{D}_{\Theta}(j,k,\ell,m)$$
	is a dense $G_\delta$ subset of $\mathscr{A}_{\Theta}(\mathbb{C}_0)$. Let us first see that $\mathscr{D}_{\Theta} \subset \mathscr{M}$. Note that if $D \in \mathscr{D}_{\Theta}$ then $$\displaystyle D \in \bigcap_{\ell \geq 1} \bigcap_{m \geq 2} \mathscr{D}_{\Theta}(1,1,\ell,m),$$ which implies $\lim\limits_{N \to \infty} A_N(D_{\lambda}, \delta_m) = \infty$ for every $m \geq 2$ and $\lambda=1$. Since in this case  $D_{\lambda}=D$, we conclude that $\sigma_a(D) \geq \delta_m \xrightarrow{m \to \infty} \dfrac{1}{2}$ and therefore $D \in \mathscr{M}$.
	
	To complete the proof, take $D \in \mathscr{D}_{\Theta}$ and let $\mathcal{A}(\{D\})$ be the subalgebra of $\mathscr{A}_\Theta(\mathbb{C}_0)$ generated by $D$. We know that $\mathscr{A}_\Theta(\mathbb{C}_0)$ is algebraically closed, so we have to see that  $\mathcal{A}(\{D\}) \subset \mathscr{M} \cup \{0\}$. Note that each $\widetilde{D} \in \mathcal{A}(\{D\}) \setminus \{0\}$ may be written in the form
	$$\widetilde{D} = D_{\lambda} = \sum_{i=1}^{k_0}\lambda_{i}D^{i} \ \ \mbox{with} \ \ k_0 \in \mathbb{N}, \ \  \lambda=(\lambda_1, \ldots, \lambda_{k_0}) \in \mathbb{C}^{k_0} \mbox{ and } \lambda_{k_0} \not= 0.$$
	Choose $j_0 \in \mathbb{N}$ such that $\Vert\lambda\Vert_\infty \leq j_0$ and $\vert \lambda_{k_0}\vert \geq j_0^{-1}$. Since $\displaystyle D \in \bigcap_{\ell \geq 1}\bigcap_{m \geq 2} \mathscr{D}_{\Theta}(j_0,k_0,\ell,m)$, for every $\ell, m \in \mathbb{N}$ there is $N_{\ell,m} \in \mathbb{N}$ such that $A_{N_{\ell,m}}(D_{\lambda},\delta_m) > \ell$. This gives $\sigma_a(D_{\lambda}) \geq \delta_m$ for each $m \geq 2$, and hence $\sigma_a(\widetilde{D})=\sigma_{a}(D_{\lambda})= \dfrac{1}{2}$. This finishes the proof.
\end{proof}

Recall that if $\mathcal{A}$ is a complex commutative algebra, a subset $X = \{x_k : k \in \mathbb N\}$ of $\mathcal{A}$ is  \emph{algebraically independent} whenever the following holds: for every $N \in \mathbb N$, if a polynomial   $Q \in \mathbb{C}[z_1, \ldots, z_N]$  satisfies that $Q(x_{1}, \ldots, x_{N})=0$, then $Q$ must be $0$.

\begin{lemma} \label{lem-alg.ind.}
	Let $\{\Theta_k\}_{k \in \mathbb N}$ be a family of pairwise disjoint infinite subsets of $\mathbb N$. We choose for each $k \in \mathbb N$ a non-constant Dirichlet series $D_k \in \mathscr{A}_{\Theta_k}(\mathbb{C}_0)$. Then $\{D_k\}_{k \in \mathbb{N}}$ is an algebraically independent subset of $\mathscr{A}(\mathbb{C}_0)$.
\end{lemma}

\begin{proof}
	From \cite[Theorem 2.5]{AroBayGauMaeNes}, the restriction of Bohr's transform $ \mathfrak{B}|_{A_u(B_{c_0})} : A_u(B_{c_0}) \to \mathscr{A}(\mathbb{C}_0)$ is an isometric algebra isomorphism. We set $f_k=\mathfrak{B}^{-1}(D_k)$ and define
	$$X := \{f_k: k\in\mathbb N\}.$$
	Hence, it suffices to show that $X$ is an algebraically independent subset of $A_u(B_{c_0})$. By means of Bohr transform, if $D_{k}(s)=\sum_{n = 1}^\infty a_{k,n}n^{-s}$, then $$f_{k}(z)=\sum_{\alpha\in\mathbb{N}_{0}^{(\mathbb{N})}}c_{\alpha}(f_{k})z^{\alpha} \ \ \mbox{for every} \ \ z \in B_{c_{00}},$$
	where $c_{\alpha}(f_k)=a_{k,n}$ with $n=\mathfrak{p}^{\alpha}$. Since $a_{k,n} = 0$ if $\mathfrak{p}_{i} \vert n$  for some $i \not\in \Theta_{k}$, we have
	\begin{align} \label{aux-def-alg-fk}
		f_{k}(z)=\sum_{\underset{supp(\alpha)\subset\Theta_{k}}{\alpha\in\mathbb{N}_{0}^{(\mathbb{N})}}}c_{\alpha}(f_{k})z^{\alpha} \ \ \mbox{for every} \ \ z \in B_{c_{00}}.
	\end{align}
	
	Take $N\in \N$ and take a polynomial $Q \in \mathbb{C}[z_1, \ldots, z_N]$ with $Q(f_{1}, \ldots, f_{N}) = 0$. Since each $f_{\ell}$ is non-constant, we can take  $\alpha^\ell \in \mathbb{N}_0^{(\mathbb{N})} \setminus \{0\}$ such that $supp(\alpha^\ell) \subset \Theta_{\ell}$ and $c_{\alpha^\ell}(f_{\ell}) \not= 0$. We set $L := \max \left(\bigcup_{\ell = 1}^N A_\ell\right)$ with $A_\ell := supp(\alpha^\ell)$. And, for each $\ell = 1, \ldots, N$, we define
	\begin{align} \label{def-pi_ell}
		\pi_{\ell}:(z_{1},\ldots,z_{L})\in\mathbb{D}^{L}\mapsto (\chi_{A_{\ell}}(j)z_{j})_{j=1}^\infty \in B_{c_0}.
	\end{align}
	
	Since $\alpha^\ell \not= 0$ and $c_{\alpha^\ell}(f_{\ell}) \not= 0$, we see from (\ref{aux-def-alg-fk}) that $f_{\ell} \circ \pi_\ell : \mathbb{D}^L \to \mathbb C$ is a non-constant holomorphic function for each $\ell = 1, \ldots, N$. Hence, by the open mapping theorem, there is $\delta > 0$ such that $B_\delta((f_{\ell} \circ \pi_\ell)(0)) \subset Im(f_{\ell} \circ \pi_\ell)$ for every $\ell = 1, \ldots, N$.
	
	As a consequence, fixed  $\lambda_\ell \in B_\delta((f_{\ell} \circ \pi_\ell)(0))$ for each $\ell=1,\dots, N$,  there exists $w_\ell = (w_{1,\ell}, \ldots, w_{L,\ell}) \in \mathbb{D}^L$ such that $f_{\ell}(\pi_\ell(w_\ell)) = \lambda_\ell$. We set
	$$v_0 := \pi_1(w_1) + \cdots + \pi_N(w_N).$$
	Since each $\pi_\ell(w_\ell)$ is supported in $\Theta_{\ell}$ and these sets are pairwise disjoint, it follows from (\ref{aux-def-alg-fk}) that
	$$f_{\ell}(v_0) = f_{\ell}(\pi_1(w_1) + \cdots + \pi_N(w_N)) = f_{\ell}(\pi_\ell(w_\ell)) = \lambda_\ell$$
	for each $\ell = 1, \ldots, N$. Thus,
	\begin{align*}
		Q(\lambda_1, \ldots, \lambda_N) = Q(f_1(v_0),\dots, f_N(v_0))=0.
	\end{align*}
	Since this holds for any $(\lambda_1,\dots,\lambda_N)$ in the open set $B_\delta(f_{1}(0)) \times \cdots \times B_\delta(f_{N}(0))$, we conclude that $Q = 0$.
\end{proof}

\begin{lemma} \label{lem-comb-1alg}
	Let $D\in\mathscr{A}_{\Theta}(\mathbb{C}_{0}) \setminus \{0\}$ be such that $\mathcal{A}(\{D\})\subset\mathscr{M}\cup\{0\}$. Then,  \begin{equation}\label{eq-lemma}
		\lambda_0 D_0 + \lambda_{1}D_{1}D+\lambda_{2}D_{2}D^{2}+\cdots+\lambda_{N}D_{N}D^{N}\in\mathscr{M}
	\end{equation}
	for all  $D_0, D_{1},\ldots,D_{N} \in \mathscr{A}_{\mathbb{N} \setminus \Theta}(\mathbb{C}_{0})$ and  $\lambda_0, \lambda_{1},\ldots,\lambda_{N}\in\mathbb{C}$ with $\lambda_{N}  D_N \not= 0$.
\end{lemma}

\begin{proof}
	For each $k \in \mathbb N$ we define $pdiv(k) := \{j \in \mathbb N : \mathfrak{p}_j | k\}$. Also, we set
	\begin{align*}D_m(s) := \sum_{n = 1}^\infty a_{m,n}n^{-s}\, ; \quad & S_M(D_m,s) := \sum_{n = 1}^{M}a_{m,n}n^{-s} \, ;\\ & T_M(D_m,s) := D_m(s) - S_M(D_m,s)
	\end{align*}
	for every $m = 0, 1, \ldots, N$ and $M\in\mathbb{N}$. Also, since $D_N\ne 0$, we can define
	$$n_0=\min\{ m \in \mathbb N : a_{m,N} \ne 0\}.$$  From the fact that $D_{N} \in \mathscr{A}_{\mathbb{N} \setminus \Theta}(\mathbb{C}_{0})$ it follows that $\mathfrak{p}_j \not |  n_0$ for every $j\in \Theta$.
	
	With this notation, we can rewrite the sum in \eqref{eq-lemma} as
	\begin{align*}
		\sum_{m=0}^N \lambda_m D_m D^m  = \widetilde{D} + {\sum_{m=0}^{N}\lambda_ma_{m,n_{0}}n_{0}^{-s}D^m}.
	\end{align*}
	where $$\widetilde{D} = {\sum_{m=0}^{N-1}\lambda_m S_{n_{0}-1}(D_m, \cdot)D^m+\sum_{m=0}^{N}\lambda_mT_{n_0}(D_m, \cdot)D^m}.$$
	First, note that $\widehat{D}=\sum_{m=0}^{N}\lambda_ma_{m,n_{0}}n_{0}^{-s}D^m$ is nonzero since it is, up to multiplication by $n_0^{-s}$, a nonzero one-variable polynomial applied to a nonconstant holomorphic function. Also, $\widehat{D}$ belongs  to $\mathscr{M}$ by hypothesis (multiplication by	$n_0^{-s}$ is not a problem).
	
	It is clear that all the nonzero terms in $\widehat{D}$ correspond to  $(n_0 \,k)^{-s}$ where $\mathfrak{p}_j | k$ only for $j\in \Theta$.
	On the other hand, since each $D_m$ belongs to $\mathscr{A}_{\mathbb{N} \setminus \Theta}(\mathbb{C}_{0})$, the coefficients of $\widetilde D$ corresponding to such $(n_0 \,k)^{-s}$ are all zero. 	This implies $$\sigma_{a}\left(\sum_{m=0}^N \lambda_m D_m D^m\right) \geq \sigma_{a}\left(\sum_{m=1}^{N}\lambda_m a_{m,n_{0}}n_{0}^{-s}D^m\right)=\frac{1}{2},$$
	which completes the proof.
\end{proof}

\begin{proof}[\underline{Proof of Theorem \ref{thm-Gdelta}}]
	The existence of a dense $G_\delta$ set follows from Proposition \ref{pps1} taking  $\Theta = \mathbb N$. To prove the algebrability part, we take $\{\Theta_k : k \in \mathbb N\}$ an infinite family of disjoint arithmetic progressions. It follows from Proposition \ref{pps1} that for each $k \in \mathbb N$ there is $D_k \in \mathscr{A}_{\Theta_k}(\mathbb{C}_0) \setminus \{0\}$ such that $\mathcal{A}(\{D_k\}) \subset \mathscr{M} \cup \{0\}$. Moreover, by Lemma \ref{lem-alg.ind.}, the set $\{D_k : k \in \mathbb N\}$ is algebraically independent. To complete the proof, it suffices to show that the subalgebra generated by $D_k$'s is contained in $\mathscr{M} \cup \{0\}$.
	
	Let $Q \in \mathbb{C}[z_1, \ldots, z_N]$ be a nonzero polynomial without constant term. Let $M$ be the maximum exponent of the variable $z_N$ in $Q$. We may suppose that $M\ge 1$ (if not, we consider $Q$ as a polynomial in less variables).  We then can write
	\begin{align*}
		Q(D_{1}, \ldots, D_{N}) = \sum_{m=0}^{M} L_m(D_{1}, \ldots, D_{{N-1}}) \ D_{{N}}^m,
	\end{align*}
	for apropriate $L_m \in \mathbb{C}[z_1, \ldots, z_{N-1}]$, $m=1,\dots,M$, with $L_M\ne 0$.
	Note that $L_m(D_{k_1}, \ldots, D_{k_{N-1}}) \in \mathscr{A}_{\mathbb{N}\setminus \Theta_{k_N}}(\mathbb{C}_0)$ for each $m$. Also, since  $D_{1}, \ldots, D_{{N-1}}$  are algebraically independent, we have $L_{M}(D_{1}, \ldots, D_{{N-1}}) \not=0$ and the result follows from Lemma \ref{lem-comb-1alg}.
\end{proof}

\section{The sets $\mathscr{D}_\Theta(j, k, \ell, m)$} \label{sec-techlem}

In this section we prove that the sets $\mathscr{D}_\Theta(j, k, \ell, m)$ are open and dense in $ \mathscr{A}_\Theta(\mathbb{C}_0)$, a fact used in the previous sections. These facts are shown in Lemmas \ref{lem-alg02} and \ref{lem-alg01} below.

\begin{lemma}\label{lem-alg02}
	The set $\mathscr{D}_{\Theta}(j,k,\ell,m)$ is open in $\mathscr{A}_{\Theta}(\mathbb{C}_{0})$.
\end{lemma}
\begin{proof}
	Let $j,k,\ell\in \mathbb{N}$ and $m \geq 2$ be fixed, and let $(D_q)_{q=1}^\infty$ be a sequence in $\mathscr{A}_{\Theta}(\mathbb{C}_0) \setminus \mathscr{D}_{\Theta}(j,k,\ell,m)$ converging to $D_{0}$ in $\mathscr{A}_{\Theta}(\mathbb{C}_0)$. Since $D_q \not\in \mathscr{D}_{\Theta}(j,k,\ell,m)$ for each $q$, there exists a sequence $(\lambda_q)_q = (\lambda_{q,1}, \ldots, \lambda_{q,k})_q \in \mathbb{C}^k$ satisfying
	\begin{eqnarray} \label{eq1}
		\Vert\lambda_{q} \Vert_{\infty}\leq j; \ \ \ \left\vert\lambda_{q,k}\right\vert \geq j^{-1}
		\ \ \mbox{and} \ \ \sup_{N \in \mathbb{N}}A_N((D_q)_{\lambda_q},\delta_m) \leq \ell.
	\end{eqnarray}
	By passing to a subsequence if necessary, we may suppose that
	$$\lambda_q \to \lambda = (\lambda_1, \ldots, \lambda_k) \in \mathbb{C}^k  \ \ \mbox{as} \ \ q \to \infty.$$
	We clearly have
	\begin{eqnarray} \label{eq2} \Vert\lambda\Vert_{\infty}\leq j \ \ \mbox{and} \ \  \left\vert\lambda_{k}\right\vert \geq j^{-1}.
	\end{eqnarray}
	Note that the mappings
	\begin{align*}
		(D,\gamma) \in \mathscr{A}_{\Theta}(\mathbb{C}_0) \times \mathbb{C}^k  & \mapsto D_{\gamma} \in \mathscr{A}_{\Theta}(\mathbb{C}_0)  \ \quad \mbox{and}  \\
		D \in \mathscr{A}_{\Theta}(\mathbb{C}_0) & \mapsto A_N(D,\delta_m) \in \mathbb{C}
	\end{align*}
	are continuous, so it follows from (\ref{eq1}) that
	\begin{eqnarray} \label{eq3}
		\sup_{N \in \mathbb{N}} A_N((D_{0})_\lambda,\delta_m) \leq \ell.
	\end{eqnarray}
	From (\ref{eq2}) and  (\ref{eq3}) we conclude that  $D_{0} \in \mathscr{A}_{\Theta}(\mathbb{C}_0) \setminus \mathscr{D}_{\Theta}(j,k,\ell,m)$ and so $\mathscr{D}_{\Theta}(j,k,\ell,m)$ is an open subset in $\mathscr{A}_{\Theta}(\mathbb{C}_0)$.
\end{proof}

For each Dirichlet series $D = \sum_{m=0}^\infty \sum_{\Omega(n) = m} a_n n^{-s}$ we define
\begin{align*}
	\widetilde{\Omega}(D) :=  & \{m \in \mathbb{N} :  D^{(m)} \not= 0\}  \\
	\,  = &  \{m \in \mathbb{N} : \mbox{there is } n \in \mathbb{N}, \Omega(n)=m, a_n \not= 0\}.
\end{align*}

We state some simple properties in the following lemma.
\begin{lemma} \label{lem-Ome}
	For every $D_m \in \mathfrak{D}_m$ and $D_n \in \mathfrak{D}_n$, we have:
	
	$(i)$ $D_m \cdot D_n \in \mathfrak{D}_{m+n}$;
	
	$(ii)$ $\min(\widetilde{\Omega}(D_m + D_n)) = \min\{m, n\}$ if $m \not= n$;
	
	$(iii)$ $\max(\widetilde{\Omega}(D_m + D_n)) = \max\{m, n\}$ if $m \not= n$.
\end{lemma}

\begin{lemma} \label{lem-alg01}
	The set $ \mathscr{D}_{\Theta}(j, k, \ell, m) $ is dense in $ \mathscr{A}_{\Theta}(\mathbb{C}_0)$.
\end{lemma}
\begin{proof}
	Let $j,k,\ell$ and $m \geq 2$ be fixed natural numbers. Let $D_1 : = \sum_{i = 0}^r \sum_ {\Omega(n) = i} a_n n^{-s} $ be a Dirichlet polynomial such that $a_{n}=0$ if $\mathfrak{p}_{i}\vert n$ for some $i\not\in\Theta$ and let $ \varepsilon> 0 $.
	
	Let us take, using  Lemma \ref{Lemma-A} and normalizing, some $D_2 \in \mathscr{H}_\infty^{m + r +1}\cap\mathscr{A}_{\Theta}(\mathbb{C}_{0}) $ such that
	$$ \| D_2 \|_{\infty} \leq \frac{1}{2} \ \ \mbox {and} \ \ \lim_{N \to \infty} A_N(D_2, \delta_m) = \infty.$$
	By Newton's Binomial Formula,
	\begin{eqnarray} \label{Eq-power}
		\left(1+\dfrac{D_2}{k}\right)^k = 1 + D_2 + D_3 \ \ \mbox{with} \ \ D_3: = \displaystyle \sum_{\ell = 2}^{k} \binom{k}{\ell} k^{-\ell} D_2^{\ell}.
	\end{eqnarray}
	Note that $\min(\widetilde {\Omega}(D_3))> m + r$ if $k \geq 2$ and $D_3 = 0$ otherwise. Since for $D_3 = 0$ the proof follows easier, we assume $k \geq 2$ from now on. We set
	\begin{align} \label{def-of-w}
		w := \max_{0 \leq i \leq k-1}\{(k-2)(m+r), rk + mi\} + 1.
	\end{align}
	Now, let $D : \mathbb{C}_0 \to \mathbb{C}$ be the Dirichlet polynomial given by
	\begin{eqnarray} \label{def-of-f}
		D(s) = D_1(s) + (2^w)^{-s} D_4(s) \ \ \mbox{with} \  \ D_4(s) := \dfrac{\varepsilon}{2} \left(1 + \dfrac{D_2(s)}{k}\right).
	\end{eqnarray} It is easy to see that $\|D - D_1\|_\infty <\varepsilon$. To complete the proof it remains to show that $D \in \mathscr{D}_{\Theta}(j,k,\ell,m)$. Note that for $q < k$ we have
	\begin{align*}
		\max(\widetilde{\Omega}(D^q)) &= \max(\widetilde{\Omega}([D_1 + (2^w)^{-s} D_4]^q)) \\
		&= \max(\widetilde{\Omega}(2^{-wqs}D_2^{q})) = wq + q(m+r) \tag{\mbox{from Lemma \ref{lem-Ome}}} \\
		&\leq w(k-1) + (k-1)(m+r) < wk + m + r \tag{\mbox{from (\ref{def-of-w})}},
	\end{align*}
	which implies (by Lemma \ref{lem-Ome})
	\begin{eqnarray} \label{conclusion-dens-1}
		\max(\widetilde{\Omega}(\lambda_{1}D+\cdots+\lambda_{k-1}D^{k-1}))< wk + m + r
	\end{eqnarray}
	for every $ \lambda = (\lambda_{1}, \ldots, \lambda_{k})$ in $\mathbb{C}^{k} $. On the other hand,
	\begin{align}\label{Eq-deg-2}
		D^k&=\left(D_1+(2^w)^{-s}D_4\right)^k
		\nonumber = \sum_{i=0}^k	{k \choose i}D_1^{k-i}\left((2^w)^{-s}D_4\right)^i
		\nonumber
		\\ &= \nonumber \sum_{i=0}^{k-1}{k \choose i}D_1^{k-i}\left((2^w)^{-s}D_4\right)^{i}+\left(\frac{\varepsilon}{2}\right)^k(2^{wk})^{-s}\left(1+\dfrac{D_2}{k}\right)^{k} \tag{\mbox{from (\ref{def-of-f})}}
		\\ &= \sum_{i=0}^{k-1}	{k\choose i}D_1^{k-i}\left((2^w)^{-s}D_4\right)^{i}+\left(\frac{\varepsilon}{2}\right)^k(2^{wk})^{-s}(1+D_2+D_3)
		\nonumber \tag{\mbox{from (\ref{Eq-power})}}\\ &=
		\underbrace{\sum_{i=0}^{k-1}{k \choose i}D_1^{k-i}\left((2^w)^{-s}D_4\right)^{i}}_{\max(\widetilde{\Omega}) < wk \ (by \ (\ref{def-of-w}))} + \underbrace{\left(\frac{\varepsilon}{2}\right)^k(2^{wk})^{-s}}_{lies \ in \ \mathfrak{D}_{wk}} +\underbrace{\left(\frac{\varepsilon}{2}\right)^k(2^{wk})^{-s}D_2}_{lies \ in \  \mathfrak{D}_{wk+m+r}} \nonumber
		\\ \label{conclusion-dens-2} & \, \qquad \, \qquad + \underbrace{\left(\frac{\varepsilon}{2}\right)^k(2^{wk})^{-s}D_3}_{\min(\widetilde{\Omega}) > wk + m + r}.
	\end{align}
	Finally, taking $\lambda = (\lambda_1, \ldots, \lambda_k)$ in $\mathbb{C}^k$ satisfying
	$$\|\lambda\|_\infty\leq j \ \ \mbox{and} \ \ \left|\lambda_{k}\right| \geq j^{-1},$$
	we obtain by using (\ref{conclusion-dens-1}) and (\ref{conclusion-dens-2}) that
	\begin{align*}
	A_N(D_{\lambda}, \delta_m) &\geq \left(\frac{\varepsilon}{2}\right)^k \ 2^{-wk\delta_m} \left\vert\lambda_{k}\right\vert A_N(D_2, \delta_m) \\ &\geq \left(\frac{\varepsilon}{2}\right)^k \ 2^{-wk\delta_m} j^{-1} A_N(D_2, \delta_m) \xrightarrow{N \to \infty} \infty.
	\end{align*}
	This shows that $D\in \mathscr{D}_{\Theta}(j,k,\ell,m)$, and the proof is complete.
\end{proof}

\bibliographystyle{amsplain}

\end{document}